\documentclass[leqno,12pt]{article}

\usepackage{epsfig}

\usepackage{amsmath}
\usepackage{amscd}
\usepackage{amsopn}
\usepackage{amsthm}
\usepackage{amsfonts,amssymb}
\usepackage{amsfonts,bbm}
\usepackage{latexsym}

\makeatletter
\@addtoreset{equation}{section}
\makeatother

\newtheorem{lemma}{LEMMA}[section]

\newtheorem{corollary}[lemma]{COROLLARY}
\newtheorem{theorem}[lemma]{THEOREM}
\newtheorem{remark}[lemma]{REMARK}
\newtheorem{remarks}[lemma]{REMARKS}

\newcommand{\real}{\mathbbm{R}}
\newcommand{\nat}{\mathbbm{N}}

\newcommand{\ganz}{\mathbbm{Z}}

\newcommand{\limn}{\lim_{n \to \infty}}

\renewcommand{\a}{\alpha}

\newcommand{\g}{\gamma}

\newcommand{\vp}{\varphi}
\newcommand{\ve}{\varepsilon}

\newcommand{\on}{\quad\text{ on }}
\newcommand{\und}{\quad\mbox{ and }\quad}
\newcommand{\inv}{^{-1}}
\newcommand{\ov}{\overline}

\newcommand{\V}{\mathcal V}  
\newcommand{\W}{\mathcal W}  
\newcommand{\C}{\mathcal C}

\renewcommand{\H}{{\mathcal H}}
\newcommand{\B}{\mathcal B}
\renewcommand{\L}{\mathcal L}
\newcommand{\N}{\mathcal N}

\newcommand{\M}{\mathcal M}

\newcommand{\U}{{\mathcal U}}

\newcommand{\Po}{\mathcal P}

\newcommand{\supp}{\operatorname*{supp}}

\newcommand{\itemframe}%
{\setlength{\parskip}{10pt}\begin{enumerate} \setlength{\topsep}{10pt}%
\setlength{\itemsep}{15pt}\setlength{\parsep}{5pt}}

\newcommand{\vx}{\ve_x}

\newcommand{\uc}{{U^c}}
\newcommand{\vc}{{V^c}}

\newcommand{\ureg}{\partial_{\mbox{\rm\small reg}}U}
\newcommand{\uirr}{\partial_{\mbox{\rm \small irr}} U}

\newcommand{\hypo}{{}^\ast \H(U)}

\title{On Evans' and Choquet's theorems for polar sets}
\author{Wolfhard Hansen and Ivan Netuka\thanks{Supported by CRC 1283 of the German Research Council (DFG).
    We thank our colleague A.\,Grigor'yan for  helping us with  useful illustrations.}}
\date{}

\begin{document}
\maketitle
  



\section{Introduction and main results}\label{intro}

  By classical results of  G.C.\,Evans and G.\,Choquet on ``good kernels $G$ in potential theory'', 
  for every  polar $K_\sigma$-set $P$,  there exists a  finite measure~$\mu$ on~$P$ such that $G\mu=\infty$
  on $P$,  and a set $P$ admits a  finite measure $\mu$ on $P$ such that $\{G\mu=\infty\}=P$
  if and only if $P$ is a polar $G_\delta$-set.
  
 We recall that Evans' theorem yields the solutions of the generalized Dirichlet problem for   open sets  
 by the Perron-Wiener-Brelot method using only \emph{harmonic} upper and lower functions
 (see \cite{cornea} and Corollary \ref{evans-corollary}). 

In this paper we intend to show that such results  can be obtained,
 by elementary ``metric'' considerations
 and without using any potential theory,
 for general kernels~$G$  locally satisfying
 \begin{equation*}
   G(x,z)\wedge G(y,z)\le C G(x,y).\footnote{We write $a\wedge b$
     for the minimum and $a\vee b$ for the maximum of $a$ and $b$.} 
 \end{equation*} 
The particular case $G(x,y)=|x-y|^{\a-N}$ on $\real^N$, 
$2<\a<N$, solves a long-standing  open problem
(see \cite[p.\,407, III.1.1]{landkof}).

 {\bfseries ASSUMPTION.}  {\sl
     Let $X$ be   a~locally compact  space  with countable base and let $G\colon X\times X\to [0,\infty]$
  be Borel measurable,  $G>0$, such that  the following holds:
 \begin{itemize}
 \item[\rm (1)]
   For every $x\in X$,    $\lim_{y\to x} G(x,y)=G(x,x)=\infty$ and $G(x,\cdot)$
   is bounded outside every neighborhood of $x$. 
 \item[\rm (2)]
   $G$ has the  local triangle property.
 \end{itemize}
}  

We recall that $G$ is said to have  the \emph{triangle property}    if, for some  $C>0$, 
  \begin{equation}\label{TP}
    G(x,z)\wedge G(y,z)\le C G(x,y) \qquad\mbox{ for all }x,y,z\in X,
       \end{equation}
  and that  $G$ has the \emph{local triangle property}  if every point in $X$ admits an open neighborhood $U$
    such that $G|_{U\times U}$ has the triangle property.

   It is well known that $G$ has the triangle property if and only if there exist a~metric~$d$ on~$X$ and
   $\g,C\in (0,\infty)$ such that, for all $x,y\in X$,
   \begin{equation}\label{triangle}
   C\inv  d(x,y)^{-\g}\le G(x,y)\le C d(x,y)^{-\g}, 
   \end{equation}
  where, by assumption (1),  every such metric yields the initial topology on $X$.
   
   To be more explicit we observe that (\ref{TP}) means that $\tilde G(x,y):=G(y,x)\le C G(x,y)$ (take $z=x$) and that
   $\rho:=G\inv+\tilde G\inv$ is a \emph{quasi-metric} on $X$, that is, for some $C>0$,
   $\rho(x,y)\le C(\rho(x,z)+\rho(y,z))$
   for all    $x,y,z\in X$.
   A metric $d$ satisfying~(\ref{triangle}) is then obtained fixing $\g\ge 2 \log_2 C$ 
   and defining
   \begin{equation*}
     d(x,y):=\inf\{\sum\nolimits_{0< j<n} \rho(z_j,z_{j+1})^{1/\g}\colon
     n\ge 2,\,  z_1=x, z_n=y,\,   z_j\in X\}
     \end{equation*}
     (see \cite[Proposition 14.5]{heinonen}).
     Conversely, every $G$ satisfying~(\ref{triangle})   has the
     triangle property, since $d(x,z)\vee d(y,z)\ge d(x,y)/2$. More generally, (\ref{TP})
     holds~if
     \begin{equation*}
           g\circ d_0\le G\le c g\circ d_0
     \end{equation*}
     for some metric $d_0$ for $X$ and
    a~decreasing function $g$  on $[0,\sup\{d_0(x,y)\colon x,y\in X\}]$
    satisfying $0<g(r/2)\le cg(r)<\infty$ for $r>0$ and  $\lim_{r\to 0}g(r)=g(0)=\infty$.

     In particular, the local triangle property holds for the classical Green function not only for domains
     in $\real^N$, $N\ge 3$, 
     but also on domains $X$ in $\real^2$
     such that $\real^2\setminus X$ is not polar ($\log (2/r)\le 2\log(1/r)$ for $0<r\le 1/2$) and  as well for   Green
     functions associated with very general  L\'evy processes (see \cite{hansen-liouville-wiener}). 

 For every  Borel set $B$ in $X$, let $\M(B)$  denote the set of all \emph{finite} positive Radon measures~$\mu$ on~$X$
 such that  $\mu(X\setminus B)=0$, and let $\M_ \eta(B)$, $\eta>0$, be the set of all measures $\mu\in\M(B)$ such that
 the total mass $\|\mu\|:=\mu(B)$ is at most $\eta$. Let 
   \begin{equation*}
    G\mu(x):=\int G(x,y)\,d\mu(y), \qquad  \mu\in \M(X), \, x\in X.
  \end{equation*}
  By assumption (1), for every $\mu\in\M(X)$,
  \begin{equation}\label{gmu-finite}
    G\mu<\infty \on X\setminus \supp(\mu).
    \end{equation}

  For every $A\subset X$, let
  \begin{equation*}
    c^\ast(A):=\inf\{\|\mu\|\colon \mu\in \M(X), \,  G\mu\ge 1\mbox{ on }A\}.
  \end{equation*}
  We observe that $c^\ast(A)=0$ if and only if there exists $\mu\in \M(X)$ such that $G\mu=\infty$ on $A$.
  Indeed, if there are $\mu_n \in\M(X)$, $n\in\nat$,  such that $\|\mu_n\|< 2^{-n}$ and $G\mu_n\ge 1$ on $A$, then
  obviously $\mu:=\sum_{n\in\nat} \mu_n\in \M_1(X)$ and $G\mu=\infty$ on $A$. Conversely, if~$\mu\in \M(X)$
  with $G\mu=\infty$ on $A$, then $\nu_n:=(1/n)\mu$ satisfies $G\nu_n=\infty\ge 1$ on $A$ and~$\limn \|\nu_n\|=0$,
  hence $c^\ast(A)=0$. 

  As already indicated, the main results of this paper are the next two theorems obtained by G.C.\,Evans
 \cite{evans}   in the classical case (where $c^\ast(P)$ is
  the outer capacity of~$P$,  and $P$ is polar if and only if $c^\ast(P)=0$, cf.\ \cite[Corollary 5.5.7]{AG})  and 
   G.\,Choquet for ``good kernels in potential theory'' \cite{choquet-cr,choquet-gd}.

  \begin{theorem}\label{main-first}
    Let $P$ be an $F_\sigma$-set in $X$,   
    $P=\bigcup_{m\in\nat} A_m$ with closed sets~$A_m$.
Then $c^\ast(P)=0$ if and only if there is  a measure $\mu\in \M(P)$ with
       $G\mu=\infty$ on $P$.
     \end{theorem}

   \begin{corollary}\label{p0-evans}
        Let $P $ be  an $F_\sigma$-set in~$X$ with $c^\ast(P)=0$.
     Let $P_0\subset P$ be countable and $A_m$, $m\in \nat$,  be closed sets such that  $\bigcup_{m\in\nat} A_m=P$
     and every intersection $P_0\cap A_m$ is dense in $A_m$.
     
 Then there is a~measure  $\mu\in \M(P_0)$  such that    $G\mu=\infty$ on $P$.
     \end{corollary}

   \begin{theorem}\label{main-second}
Let  $P\subset X$ and let $P_0$ be a~countable dense set in $P$.  The following are equivalent:
    \begin{itemize}
    \item[\rm(i)] $P$ is a $G_\delta$-set and $c^\ast(P)=0$.
            \item[\rm (ii)] There exists $\mu\in\M(P)$ such that $\{G\mu=\infty\}=P$.
        \item[\rm (iii)] There exists $\mu\in\M(P_0)$ such that $\{G\mu=\infty\}=P$.
             \end{itemize}
    \end{theorem}

    \begin{remark}{\rm
        Let us note that J.\,Deny \cite{deny} had made a step in the direction of Choquet's result in proving that,
        for every  $G_\delta$-set $P$ in $\real^N$ which is polar (with respect to classical potential theory),
        there exists a~measure~$\mu$ on~$\real^N$ such that $\{G\mu=\infty\}=P$.
      }
      \end{remark}

    \section{Case, where  $G$ has the triangle property on $X$} \label{special}
    
 Let us  consider first the case, where $G$ has the triangle property on $X$,  and therefore
 $  C\inv  d^{-\g}\le G\le C d^{-\g}$  for some metric $d$ for $X$ and $\g,C\in (0,\infty)$.
 Defining $\tilde G:=d^{-\g}$ we then have
$   C\inv \tilde G\mu \le G\mu\le C \tilde G\mu$,
 and hence $\{\tilde G\mu=\infty\}=\{G\mu=\infty\}$ for every  $\mu\in \M(X)$.
 So the implications of Theorems \ref{main-first} and Theorem \ref{main-second} follow immediately
if they hold for $\tilde G$. Thus we may and shall assume in this section without loss of generality that
        \begin{equation*}
      G(x,y)=d(x,y)^{-\g}, \qquad x,y\in X.
    \end{equation*}
    Then, for every $\mu\in\M(X)$, the ``potential'' $G\mu$  
    is lower semicontinuous on $X$ and
    continuous outside the support of $\mu$. Moreover,
    if $A$ and $B$ are Borel sets in  $X$ such that $ d(A,B):=\inf\{d(x,y)\colon x\in A, y\in B\}>0$, then
    \begin{equation*}
      G\mu(x) \le   d(A,B)^{-\g}\|\mu\| \qquad\mbox{ for all  $x\in A$ and $\mu\in \M(B)$.}
      \end{equation*}

      The key for       Theorems \ref{main-first} and \ref{main-second} in our setting
      will be Lemmas \ref{key}, \ref{key-0}  and \ref{key-finite}.
 
       \subsection{Proof of Theorem \ref{main-first}}

       \begin{lemma}\label{key}
         Let  $\emptyset \ne A \subset X$ be closed, let $A_0\subset A$  be a Borel set which is dense in $A$,
         and let $\mu\in \M(X)$.
       \begin{itemize} 
       \item[\rm  (a)]
           If $\mu(A)=0$, there exists $\nu\in \M(A_0)$ such that
           \begin{equation}\label{eins}
             \|\nu\|=\|\mu\|  \und \mbox{$ G\nu\ge  3^{-\g} G\mu$ on $A$.}
             \end{equation}
           \item[\rm (b)]   There exists $\nu\in\M(A)$ such that {\rm(\ref{eins})} holds.
             \end{itemize} 
\end{lemma}
 
\begin{proof}  (a)  Assuming $\mu(A)=0$, we consider ``shells'' $S(A,r)$ around $A$ defined by
  \begin{equation*}
    S(A,r):=\{y\in X\colon 3r\le d(A,\{y\}) < 4r\}, \qquad r>0.
    \end{equation*} 
  Since  $X\setminus A$ is obviously covered by the   ``shells''
  $S(A,(4/3)^k)$, $k\in\ganz$,  we may suppose without loss of generality that
  $\mu$ is supported by some $S(A,r)$, $r>0$.

  For $x\in X$ and $R>0$, let $B(x,R):=\{y\in X\colon d(y,x)<R\}$. There is a~sequence~$(x_n)$
  in~$A_0$ such that $A$ is covered by the open balls  $B(x_n,r)$, $n\in\nat$,
  and hence $S(A,r)$   is the union of the sets
  \begin{equation*}
    A_n:=S(A,r)\cap B(x_n,5r).
    \end{equation*} 
  So there exist $\mu_n\in\M(A_n)$, $n\in\nat$, such that $\sum_{n\in\nat} \mu_n=\mu$.

     For the moment, let us fix $n\in\nat$ and  $x\in A$. If $y\in A_n$,
    then  $d(y,x_n)<5r$ and  $3r\le d(x,y)$,   hence
    $d(x,x_n)\le d(x, y)+d(y,x_n) < 3 d(x,y)$. 
      \begin{center}
       \includegraphics[width=8cm]{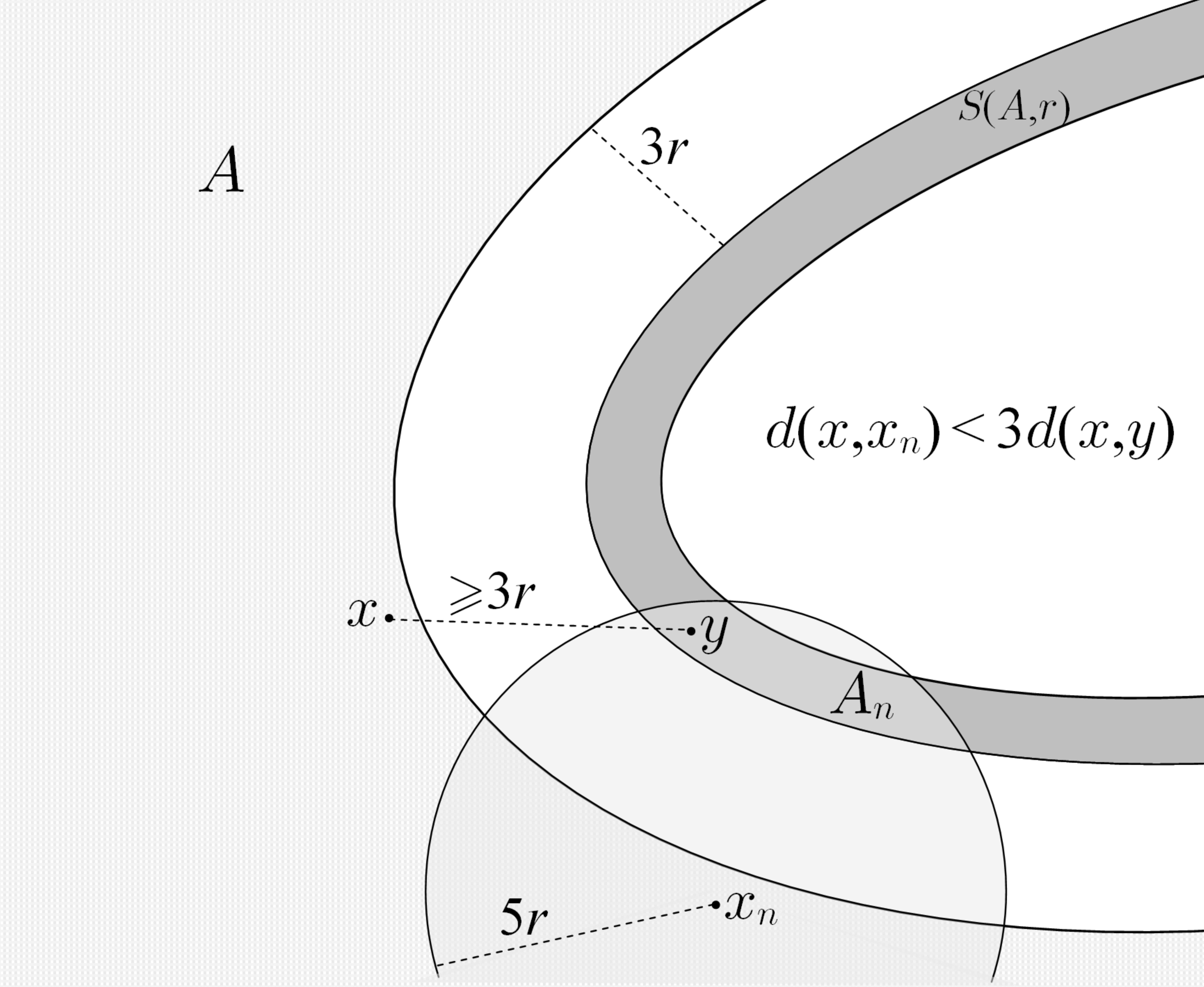}\\
 {\small Figure 1. ``Sweeping'' of $\mu_n$ from $A_n$ to $x_n$}
\end{center}
So 
$d(x,x_n)^{-\g}>3^{-\g} d(x,y)^{-\g}$.  
      Integrating with respect to $\mu_n\in \M(A_n)$ we see that     
\begin{equation*} 
      \|\mu_n\|G(x,x_n)\ge 3^{-\g}   G\mu_n(x). 
\end{equation*} 
Clearly,  $\nu:=\sum_{n\in\nat}  \|\mu_n\|\,\delta_{x_n}\in \M(A_0)$ (where $\delta_{x_n}$
denotes Dirac measure at~$x_n$) and the measure $\nu$  satisfies (\ref{eins}).

(b)    Let $\mu':=1_{A^c}\mu$ and $\mu'':=1_{A} \mu$.
     By (a), there exists $\nu'\in \M(A)$
     such that $\|\nu'\|=\|\mu'\|$ and $G\nu'\ge 3^{-\g} G\mu'$ on $A$. Since $\mu=\mu'+\mu''$, we obtain
     that  $\nu:=\nu'+\mu''\in \M(A)$ and (\ref{eins}) holds.
       \end{proof}

       \begin{remark}
         {\rm   
           It is easily seen that, for every $\ve>0$, we can get  (a) in Lemma~\ref{key}  with $(2+\ve)^{-\g}$
           in place of $3^{-\g}$ replacing $3r$ and $4r$ in the definition of $S(A,r)$ by
           $Mr$ and $(M+1)r$, $M$ sufficiently large. 
         }
\end{remark}

       \begin{proof}[Proof of Theorem \ref{main-first}]
                           Let $\mu\in \M_1(X)$,   $G\mu=\infty$ on $P$.
            By   Lemma \ref{key}, there are  $\nu_m\in \M_1(A_m)$ with
         $G\nu_m=\infty$ on $A_m$, $m\in\nat$. Then $\nu:=\sum_{m\in\nat} 2^{-m}\nu_m\in\M_1(P)$ and
         $G\nu=\infty$ on~$P$.
       \end{proof}

We might note that until now we did not use local compactness of $X$.       
       
\subsection{Proof of Corollary \ref{p0-evans}}

We shall need a lemma which follows from the weak$^\ast$-lower semicontinuity of the mappings
$\nu\mapsto G\nu(x)=\sup_{m\in\nat}(G\wedge m)\nu(x)$, $x\in X$, and the lower semicontinuity of the functions $G\nu$
(see  \cite[p.\,26]{brelot}   or \cite[Lemme 1]{choquet-gd}).

       \begin{lemma}\label{weak}
         Let $K,L$ be compacts in $X$, $L\subset K$. Let $\mu\in \M(K)$ and $\vp$ be a~continuous function
         on $L$ such that $G\mu>\vp$ on $L$. Then there exists a weak$^\ast$-neighborhood $\N$ of $\mu$
         in $\M(K)$ such that, for every $\nu\in \N$,  $G\nu>\vp$ on $L$.
       \end{lemma}
         
     Given a compact $K\ne \emptyset$ in $X$, a measure $\mu\in\M(K)$
       and a dense  set~$A_0$ in~$K$, we construct  \emph{approximating   measures}
        \begin{equation}\label{approx}
         \mu^{(n)}=\sum\nolimits_{1\le j\le M_n} \a_{nj} \delta_{x_{nj}}, \quad   \a_{nj}\ge 0,\ x_{nj}\in A_0,
         \quad n\in\nat,
         \end{equation} 
  in the following way: 
     Having fixed $n\in\nat$,                      
                we  take $x_1,\dots, x_M\in A_0$
         such that the balls $B_j:=B(x_j,1/n)$  cover $K$, choose  
         $\mu_j\in \M(B_j)$ with  $\mu=\sum_{1\le j\le M} \mu_j$,  and define
         $                          \mu^{(n)}:=\sum_{1\le j\le M} \|\mu_j\| \delta_{x_j}$. Of course,  $\|\mu^{(n)}\|=\|\mu\|$.

         Clearly,  the sequence  $ (\mu^{(n)})$ is weak$^{\ast}$-convergent to~$\mu$, and, for every open   
       neighborhood $W$ of $K$, the sequence $(G\mu^{(n)})$ converges to $G\mu$ uniformly on $W^c$, 
       since the functions $y\mapsto d(x,y)^{-\g}$, $x\in W^c$,  are  equicontinuous on~$K$. 
                     
       \begin{proof}[Proof of Corollary \ref{p0-evans}]
        We may assume without loss of generality that $(A_m)$ is an increasing sequence of compacts.  
        Given $m\in\nat$, there exists $\mu_m\in \M(A_m)$ such that $G\mu_m=\infty$ on $A_m$,
        by Theorem \ref{main-first}. 
        By Lemma \ref{weak} and using   approximating   discrete measures,
        we obtain
          $\nu_m\in \M_1(P_0\cap A_m)$ such that  $G\nu_{m}> 2^m$ on~$A_m$.  
           Then~$   \nu:=\sum_{m \in\nat}       2^{-m} \nu_m\in \M_1(P_0)$ and $G\nu=\infty$ on~$P$.
         \end{proof}

\subsection{Proof of Choquet's theorem}

The implication  (iii)\,$\Rightarrow$\,(ii)  holds trivially. 
   Since, for every $\mu\in \M(X)$, the function~$ G\mu$ is lower semicontinuous, and therefore
$      \{G\mu=\infty\}      =\bigcap_{n\in\nat}\{  G\mu>n\}$
is a~$G_\delta$-set, we also know that  (ii)  implies (i). 
 
Based on Lemma \ref{key}, the next two lemmas are the additional   ingredients
for  the~proof of the implication  (i)\,$\Rightarrow$\,(iii)
in our setting. They replace the potential-theoretic Lemme 3 in \cite{choquet-gd}. 
 The other steps can, more or less, be taken as in  \cite{choquet-gd}.

       A  sequence $(U_n)$ of open sets in $X$ with $\bigcup_{n\in\nat} U_n=U$ will be called
 \emph{exhaustion of~$U$} provided, for every $n\in \nat$, the closure $\ov U_n$   is a compact
 subset of $U_{n+1}$.

  \begin{lemma}\label{key-finite}     
    Let  $V\subset  X $ be   open, $\nu_0\in \M(V)$, and $M>0$.
    There  exists~\hbox{$\nu\in \M(V)$} such that  $\nu\le \nu_0$,
   $G\nu<\infty$ on $V^c$ and $G\nu>M$ on~ $\{G\nu_0>M+1\}\cap V$.        
       \end{lemma}

 \begin{proof} 
   Let us choose exhaustions $(V_n)$  and  $(W_n)$ of  $V$ and $W:=\{G\nu_0>M+1\}$, respectively.
   We~claim that there          are measures $\nu_0\ge \nu_1\ge \nu_2 \ge \dots$ such that 
         \begin{equation}\label{ind}
           G\nu_n>M+2^{-n} \mbox{ on }W , \qquad   G( \nu_{n-1}-\nu_n)<2^{-n} \mbox{ on  }     V_n.   
         \end{equation} 
           Let us fix $n\in \nat$ and suppose that we have
          $\nu_{n-1}\le \nu_0 $ such that $G\nu_{n-1}>M+2^{-(n-1)}$ on $W$ (which holds if $n=1$).
          Since  $V_m\uparrow V$ and therefore $G(1_{V_m}\nu_{n-1})\uparrow G\nu_{n-1}$ as~$m\to \infty$,
          there exists   $m>n$  such that   
           \begin{equation*} 
             B_n:= W_n\cap (V\setminus V_m)\subset W_n \setminus V_{n+1} \und \nu_n:=1_{V\setminus B_n}\nu_{n-1}
           \end{equation*}
          (see Figure 2) satisfy 
            \begin{equation}\label{crucial}
              \mbox{$\nu_0(B_n)< 2^{-n} d(W_n\setminus V_{n+1},  W_{n+1}^c\cup  V_n) ^{\g} $}              
              \mbox{  \ and \ }
              \mbox{$   G\nu_n>M+2^{-n} $ on $\ov W_{n+1}$.} 
            \end{equation}
  \begin{center}
       \includegraphics[width=8.8cm]{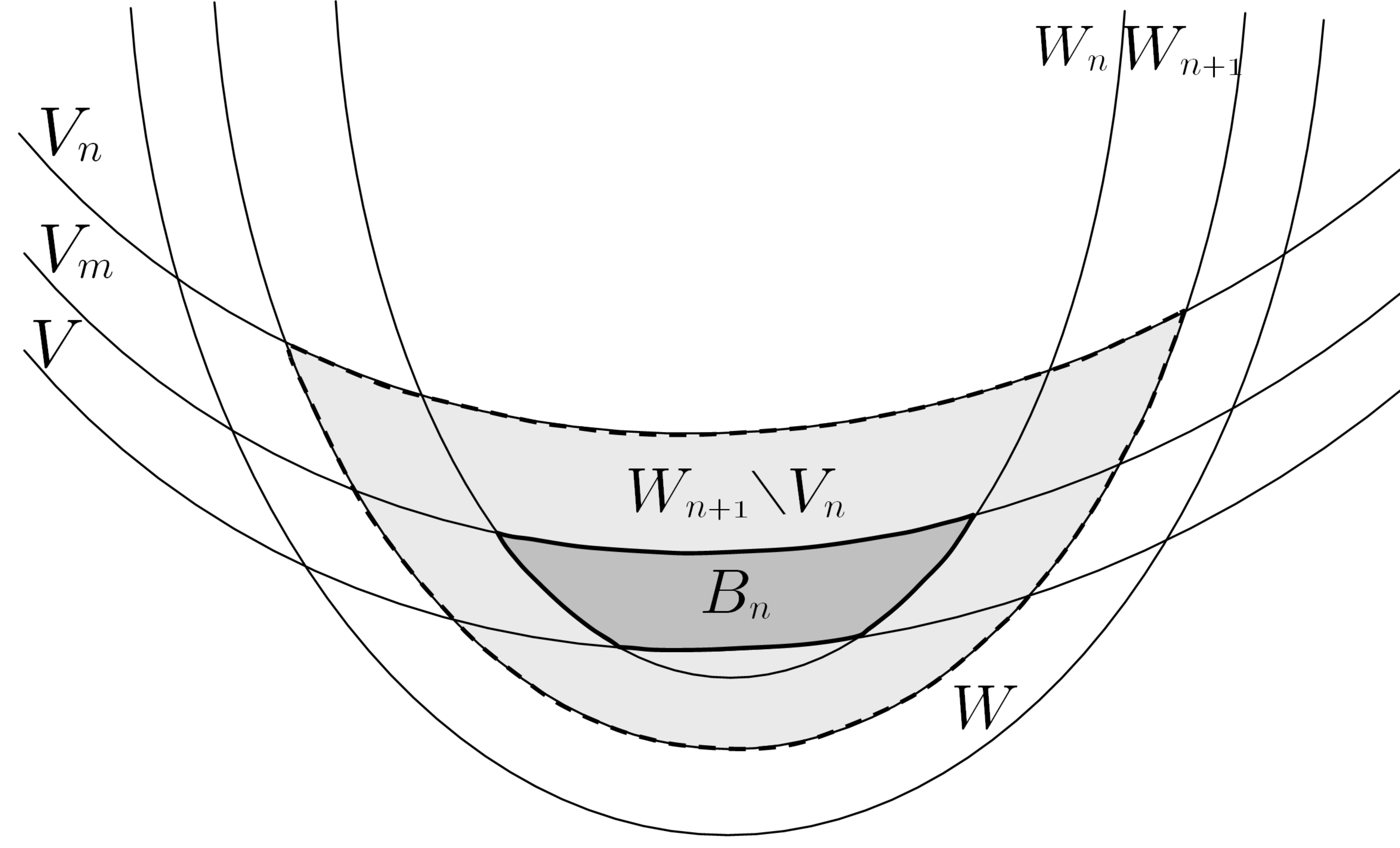}\\
       {\small Figure 2. The choice of $B_n$}
       \end{center}
         Then  $  G(\nu_{n-1}-\nu_n)\le G(1_{B_n}\nu_0)< 2^{-n}$ on $W_{n+1}^c\cup V_n$. In particular, 
         \begin{equation*}
           G\nu_n\ge  G\nu_{n-1}  -2^{-n}>M+2^{-(n-1)}-2^{-n}= M+2^{-n} \mbox{ on }W\setminus W_{n+1}.
         \end{equation*}
         Having the second inequality in~(\ref{crucial}) we conclude that~(\ref{ind}) holds.
          
          The sequence $(\nu_n)$ is decreasing to a measure $\nu$ such that, for every $n\in\nat$,
          \begin{equation*}
            G(\nu_n-\nu)=\sum\nolimits_{j=n}^\infty G(\nu_j-\nu_{j+1})
            <\sum\nolimits_{j=n}^\infty  2^{-(j+1)} = 2^{-n} \mbox{ on } V_n, 
          \end{equation*}
                    and hence $G\nu>M$ on $W\cap V_n$, by (\ref{ind}). 
                    So $ G\nu>M$ on $W\cap V$.

                    Of course, $G\nu<\infty$ on $\ov V^c\cup W^c$.  
                    By our construction, for every $n\in\nat$, the support of $\nu_n$
                    does not intersect $W_n\cap \partial V$, and hence $G\nu\le G\nu_n<\infty$ on $W_n\cap \partial V$.
                    Therefore $G\nu<\infty$ on $W\cap \partial V$, and we finally obtain that $G\nu<\infty$ on $V^c$.
            \end{proof}

         \begin{lemma}\label{key-0}
           Let $U$ be a relatively compact open set in $X$,  let $ P$ be a subset of $U$ with  $c^\ast(P)=0$, and
           let $\ve>0$.
           Then there  exist an open neighborhood $V$ of~$P$ in~$U$ and~$\mu\in\M_\ve(\ov P\cap V)$ such that
          $G\mu<\infty$ on $V^c$ and  $G\mu>2$ on $\ov P\cap V$. 
         \end{lemma}

     \begin{proof}  Let $\nu_0\in \M_\ve(X)$ with  $G\nu_0=\infty$ on $P$. Since $G(1_{U^c}\nu_0)<\infty$
         on $U$, by (\ref{gmu-finite}), we may assume that $\nu_0$ is supported by $U$. Of course,
         \begin{equation*}
        P\subset V:=\{x\in U\colon G\nu_0(x)>9^{\g+1}+1\}\subset U. 
      \end{equation*}
      By Lemma \ref{key-finite}, there exists $\nu\in\M_\ve(U)$ with $G\nu<\infty$ on $U^c$ and $G\nu>9^{\g+1}$ on $V$.
         Let $\nu_1:=1_V\nu$, $\nu_2:=1_{U\setminus \ov V}\nu$ and $\sigma:=1_{U\cap \partial V}\nu$
         so that $\nu=\nu_1+\nu_2+\sigma$.
    \begin{center}
       \includegraphics[width=6.7cm]{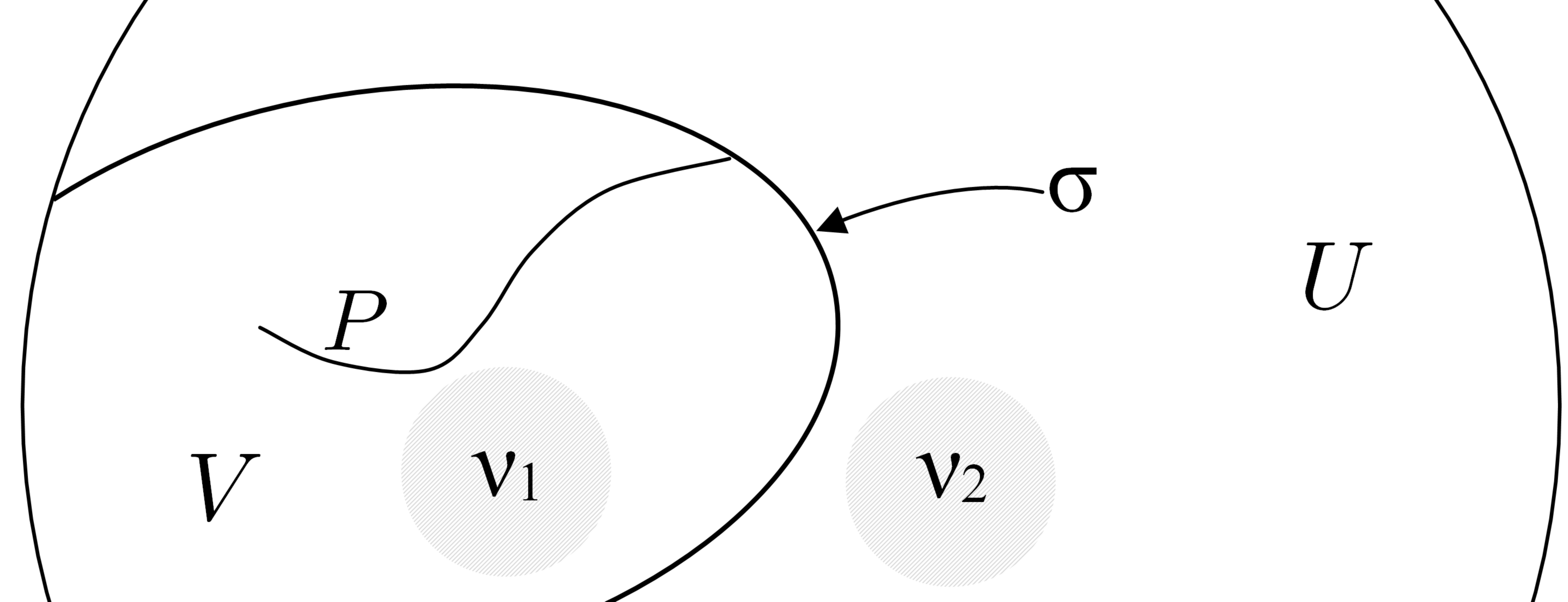}\\
 {\small Figure 3. Decomposition of $\nu$}
\end{center}

         By  Lemma \ref{key},(a), applied to $U$, $A:=\ov V\cap U$, $A_0:=V$, there exists
         $\tilde \nu_2\in \M(V)$ such that 
         \begin{equation*}
           \|\tilde \nu_2\|=\|\nu_2\| \und G\tilde \nu_2\ge 3^{-\g} G\nu_2\mbox{ on }\ov V\cap U.
         \end{equation*}
         Of course, $G\sigma<\infty$  outside the boundary $\partial V$ supporting $\sigma$. 
        Further,
         $G\sigma\le G\nu<\infty$ on $\uc$ and, by definition of $V$, on $U\setminus V$.  So~$G\sigma<\infty$ on $X$.
                  By Lusin's theorem and a~version of the continuity principle  of Evans-Vasilesco (see \cite[pp.\,97--98]{HN-(H)}),
         there are $\sigma_n\in \M(X)$ such that  each $G\sigma_n$  is continuous on $X$ and 
         $\sigma=\sum_{n\in\nat} \sigma_n$.

                  Using (\ref{approx})  and Lemma \ref{weak} with $K=L=\ov V$, we get  $\tilde\sigma_n\in\M(V)$ with  
         \begin{equation*}
         \|\tilde \sigma_n\|=  \|  \sigma_n\|   \und  G\tilde\sigma_n>G\sigma_n-2^{-n} \mbox{ on }\ov V.
           \end{equation*} 
        We define $ \tilde \sigma=\sum\nolimits_{n\in\nat} \tilde \sigma_n $ and
                     $ \tilde \nu:=\nu_1+\tilde \nu_2+\tilde \sigma$.
                     Then
                     \begin{equation*}
                       \tilde \nu\in \M_\ve(V) \und
                       G\tilde \nu\ge 3^{-\g}G\nu-1>3^{\g+2}-1>3^{\g+1}\mbox{ on }V.
                     \end{equation*}
                     
        Applying   Lemma \ref{key},(b)  to $V$, we get $\mu_0\in \M(\ov P\cap V)$ with 
           $  \|\mu_0\|=\|\tilde \nu\|\le \ve $ and $G\mu_0\ge 3^{-\g}G\tilde \nu>3$ on $\ov P\cap V$. 
           Finally, by Lemma \ref{key-finite}, we obtain a measure $\mu\le \mu_0$ such that $G\mu< \infty$ on $\vc$
           and $G\mu>2$ on $\ov P\cap V$. 
         \end{proof}

            \begin{lemma}\label{lemme-2}
            Let $ P\subset X$ such that   $c^\ast(P)=0$,  $U$ be an open neighborhood of~$P$ and~$0<\ve\le 1$. 
                               There are an open neighborhood $V$ of $P$ in $U$ and  $\mu\in \M_\ve (\ov P\cap V)$
                          such that~$G\mu>2$ on $\ov P\cap V$, $G\mu<\infty$ on $V^c$, and $G\mu<\ve$ on $U^c$.
                 \end{lemma}

                     \begin{proof}  Let $(W_n)$ be an exhaustion of $U$. Let $n\in\nat$,
                       \begin{equation*} U_n:=W_{n+1}\setminus \ov W_{n-1}, \quad
                         P_n:=P\cap U_n, \quad  \ve_n:=  2^{-n} \ve  \, (1\wedge d(U_n, W_{n-2}\cup W_{n+2}^c)^\g )
                        \end{equation*}
                        (with $W_{-1}=W_0=\emptyset$).
                        By Lemma \ref{key-0}, there exist an open neighborhood~$V_n$ of~$P_n$ in~$U_n$
                        and~$\mu_n\in\M_{\ve_n} (\ov P\cap V_n)$ such that
                      $G\mu_n>2$ on $\ov P\cap V_n$ and $G\mu_n<\infty$ on~$V_n^c$.
               \begin{center}
       \includegraphics[width=8.6cm]{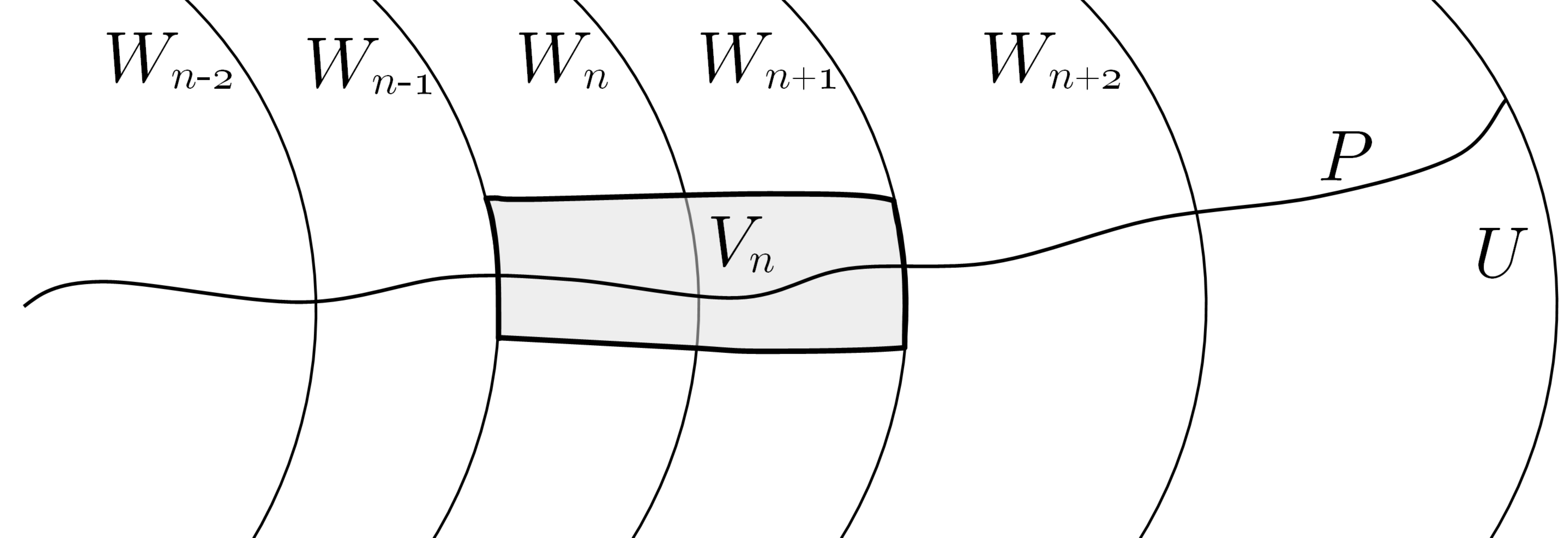}\\
   {\small Figure 4. Choice of $V_n$}
\end{center}
 By our choice of $\ve_n$, 
                      $G\mu_n\le 2^{-n}\ve $ on $W_{n-2}\cup W_{n+2}^c$.
                        It is  immediately verified that~$\mu:=\sum_{n\in\nat} \mu_n$ and $V:=\bigcup_{n\in\nat} V_n$
                        have the desired properties.
                        \end{proof}

                        We may now continue similarly as in \cite{choquet-gd}, 
                         but in a~slightly simpler way using approximating sequences instead of weak$^\ast$-neighborhoods.
                                                
                      \begin{lemma}\label{lemme-4}   
                        Let  $ P\subset X$ with $c^\ast(P)=0$ and   $P_0$   be a countable,  dense set in $P$.
                        Let $U$ be an open neighborhood of $P$ and  $\ve>0$.
                                            There exists $\nu\in \M_\ve(P_0)$ such that
                        \begin{equation*}
                        \mbox{  $G\nu >1 $ on $P$},\quad 
                        \mbox{$G\nu<\infty$       on $X\setminus P_0$}  \und   \mbox{$G\nu< \ve$ on $U^c$}.
                        \end{equation*} 
                      \end{lemma}

                      \begin{proof}[Proof {\rm (cf.\ the proof of  \cite[Lemme 2]{choquet-gd})}]
                        Let $\delta:= 1\wedge(\ve/2)$. 
                          By Lemma \ref{lemme-2}, there exist an open neighborhood $V$ of $P$ in $U$
                        and  
                        $\mu\in \M_\delta(\ov P\cap V) $ such that
                        \begin{equation}\label{previous}
                           G\mu>2 \mbox{ on } \ov P\cap V, \quad G\mu<\infty\mbox { on }V^c\und G\mu<\delta\mbox{  on }U^c.
                     \end{equation} 
                     Let $(V_k)$ be an exhaustion of $V$. For $k\in\nat$, we define (taking $V_{-1}=V_0:=\emptyset$)  
                                             \begin{equation*}
                            P_k:= P_0\cap(V_k\setminus V_{k-1})  \und W_k:= V_{k+1}\setminus \ov V_{k-2}
                          \end{equation*}
                          so that $W_k$ is an open neighborhood of $\ov P_k$, see Figure 5. 
                                         \begin{center}
       \includegraphics[width=8cm]{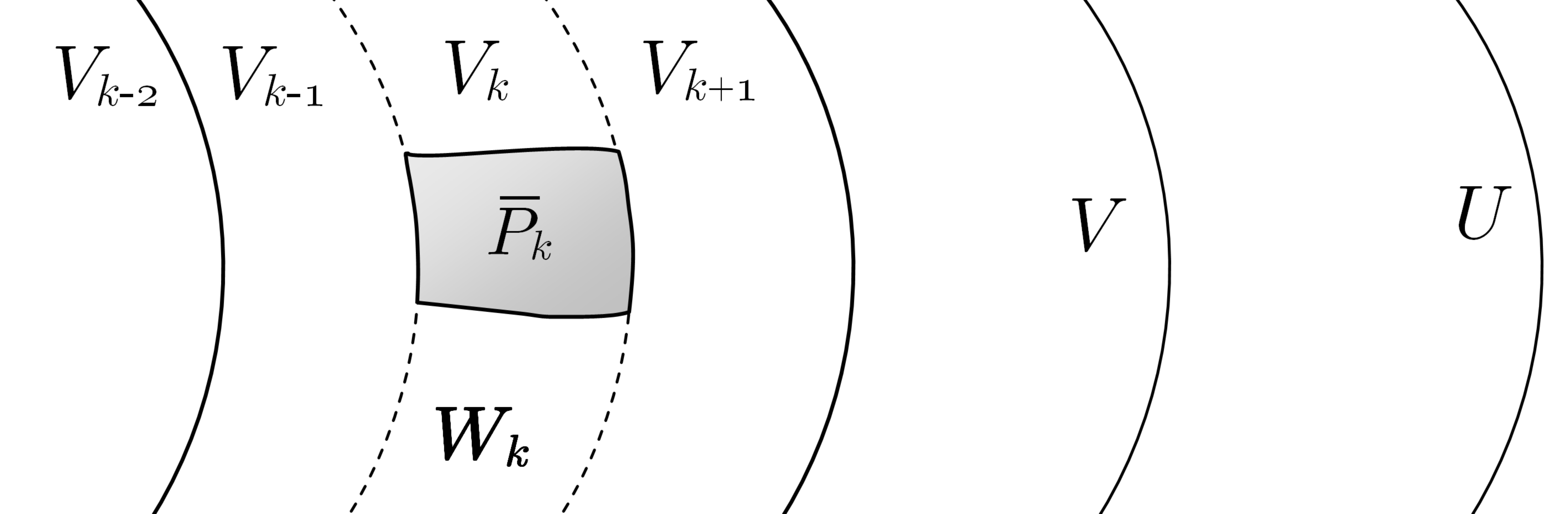}\\
    {\small Figure 5. The open neighborhood $W_k$ of $\ov P_k$}
\end{center}
Clearly, every point in
                          $\ov P\cap (V_k\setminus V_{k-1})$ is contained
                          in the closure of $P_{k-1}\cup P_k$. Hence we have $\bigcup_{k\in\nat} \ov P_k=\ov P\cap V$.
                          
                        We choose $\mu_k\in \M(\ov P_k)$  with $\mu=\sum_{k\in\nat} \mu_k$ and 
                                        approximating sequences~$(\mu_k^{(n)})_{n\in\nat}$   for~$\mu_k$ in~$\M(P_k)$, 
                          see (\ref{approx}). 
                          So, for every $k\in\nat$,  the sequence  $(\mu_k^{(n)}) $ is weak$^\ast$-convergent to $\mu_k$
                          and the sequence $(G\mu_k^{(n)}) $ converges   to $G\mu_k$ uniformly on $W_k^c$.
                          
                          For the moment, we fix $k\in\nat$.
                       There exists $l_k\in\nat$ such that, for all $n\ge l_k$, 
                                   \begin{equation}\label{est-k}
                                    |G\mu_k^{(n)}-G\mu_k|< 2^{-k}\delta \on W_k^c.
                                  \end{equation} 
                                 Let $\tau_k:=\sum_{|m-k|>1} \mu_m$.  Then $G\tau_k$ is continuous on $\ov P_k$
                                 and, by (\ref{previous}),
                             \begin{equation*} 
                                  G(\mu_{k-1}+\mu_k+\mu_{k+1})>2-G\tau_k
                                  \on   \ov P_k.\footnote{To be formally correct we have to omit, here and below, the term with
                                   subscript $k-1$ if $k=1$.}
\end{equation*} 
                                 By Lemma \ref{weak} (applied with $K:=\ov P_{k-1}\cup \ov P_k\cup \ov P_{k+1}$ and $L:=\ov P_k$),
                                 there exists $m_k\ge l_k$ such that, for all $n\ge m_k$,  
                        \begin{equation}\label{kkk}          
                          G(\mu_{k-1} ^{(n)}+\mu_k ^{(n)}+\mu_{k+1} ^{(n)} )> 2-G\tau_k \on   \ov P_k.
                        \end{equation}

                        We now define                        
                  \begin{equation*}
                    n_k:=m_{k-1}\vee  m_k\vee m_{k+1} , \quad  \nu_k :=\mu_k^{(n_k)}
                    \und \nu:= \sum\nolimits_{k\in\nat}\nu_k.
                    \end{equation*} 
                    Then $\nu\in \M (P_0)$, $\|\nu\|=\|\mu\|\le \delta$, and $G\nu<\infty$ on $V\setminus P_0$,
                          since   $\nu$ is supported by a~subset of $P_0$ having no
                            accumulation points in $V$. By (\ref{est-k}),
                            \begin{equation*}
                              G\nu\le G\mu + \sum\nolimits_{k\in\nat} |G\nu_k-G\mu_k|< G\mu+\delta \on V^c.
                            \end{equation*}
                            So, by (\ref{previous}),  $G\nu<\infty$ on $V^c$ and $G\nu<2\delta\le \ve$ on $U^c$.
                          By (\ref{kkk}), for every $k\in\nat$,
                            \begin{equation*}
                              G\nu\ge G(\nu_{k-1}+\nu_k+\nu_{k+1}+\tau_k)-\sum\nolimits_{|m-k|>1}|G\nu_m-G\mu_m|
                              > 2-\delta\ge 1
                            \end{equation*}
                            on $  \ov P_k$. Thus $G\nu>1$ on the set $\ov P\cap V$ containing $P$.
                                                  \end{proof} 
                    
                                                  \begin{proof}[Proof of {\rm (i)\,$\Rightarrow$\,(iii)} in Theorem \ref{main-second}]
  Let $(U_m)$ be a decreasing sequence of open sets in $X$ such that $P=\bigcap_{m\in\nat} U_m$.
  By Lemma \ref{lemme-4}, there are $\nu_m\in \M(P_0)$, $m\in\nat$,  such that $\|\nu_m\|\le 2^{-m}$,
  $G\nu_m> 1$ on $P $,
  $G\nu_m<\infty$ on $X\setminus P_0$ and  $G\nu_m<2^{-m}$   on~$X\setminus U_m$.
  Obviously, $\nu:=\sum_{m\in\nat}\nu_m\in \M(P_0)$ and $\{G\nu=\infty\}=P$.
  \end{proof} 

  \section{The general case}

Assuming that $G$ has the local triangle property
there is a locally finite covering of~$X$ by relatively compact open sets $U_n$ such that, for each $n\in\nat$,
the restriction of~$G$ on $U_n\times U_n$ has the triangle property.
 
 Let us consider  $ P\subset X$, a countable dense set $P_0$ in $P$,  and $\mu\in \M_1(X)$
 such that $G\mu=\infty$ on $P$. For $n\in\nat$, we introduce
 \begin{equation*}
   P_n:=P\cap U_n  \und     \mu_n:=1_{U_n}\mu\in\M_1(U_n).
 \end{equation*}
  By (\ref{gmu-finite}),  $ G(1_{U_n^c}\mu)<\infty$ on $U_n$, and hence 
  \begin{equation}\label{Gnn}
      G\mu_n=\infty\mbox{ on } P_n.
 \end{equation}

a)  If $P$ is an $F_\sigma$-set, then every $P_n$ is an $F_\sigma$-set, and 
     applying Theorem \ref{main-first} to~$U_n$ and $G|_{U_n\times U_n}$,
     we obtain $\nu_n\in\M_1(P_n)$ such that $G\nu_n=\infty$   
    on $P_n$.  Then  clearly $\nu:=\sum_{n\in\nat} 2^{-n}\nu_n\in \M_1(P)$ and
    $G\nu=\infty$ on $ P$     completing the proof of Theorem~\ref{main-first}.
    A straightforward modification yields   Corollary~\ref{p0-evans}. 
    
   b)  Let us next suppose that $P$ is a $G_\delta$-set and let $n\in\nat$. Then $P_n$ is a $G_\delta$-set
    and an application of Theorem \ref{main-second} to $U_n$ and $G|_{U_n\times U_n}$ yields 
   $      \nu_n\in\M_1(P_0\cap U_n) $ such that $\{x\in U_n: G\nu_n(x)=\infty\}=P_n$. 
            By (\ref{gmu-finite}), $G\nu_n<\infty$  on $U_n^c$, and therefore
            \begin{equation}\label{choqu}
              \{x\in X\colon G\nu_n(x)=\infty\}=P_n.
              \end{equation} 
              Obviously,      
                \begin{equation*}
        \nu:=\sum\nolimits_{n\in\nat} 2^{-n} \nu_n \in \M_1(P_0) \und G\nu=\infty\mbox{ on } P.
        \end{equation*} 
        To show that $G\nu<\infty$ outside $P$, we fix $n\in\nat$ and note  that the set $I_n$ of all $k\in\nat$
        such that $U_k\cap U_n\ne\emptyset$ is finite. Defining $\rho_n:=\sum_{k\in I_n^c} 2^{-k}\nu_k$ we know
               that $G\rho_n<\infty$ on $U_n$,   by (\ref{gmu-finite}).  Hence, using (\ref{choqu}), 
        \begin{equation*}
          G\nu=G\rho_n+\sum\nolimits_{k\in I_n} 2^{-n} G\nu_k< \infty \on U_n\setminus P.
          \end{equation*} 
          Thus $\{G\nu=\infty\}=P$.   

        c)   To complete the proof of Theorem \ref{main-second} we suppose that $\{G\mu=\infty\}=P$ and have
        to show that $P$ is a $G_\delta$-set. Let $n\in\nat$. By  (\ref{Gnn}), $\{x\in U_n\colon G\mu_n(x)=\infty\}=P_n$.
      By Theorem \ref{main-second}, applied to $U_n$ and $G|_{U_n\times U_n}$, we obtain that $P_n$ is a $G_\delta$-set.
      So there exist open neighborhoods $W_{nm}$, $m\in\nat$,  of $P_n$ in $U_n$ such that $W_{nm}\downarrow P_n$
      as~$m\to \infty$. Since $I_n$ is finite, we then easily see that   
       $    W_m:=\bigcup\nolimits_{n\in\nat} W_{nm} $
       is decreasing to $\bigcup\nolimits_{n\in\nat} P_n=P$
      as~$m\to \infty$. Thus $P$ is a $G_\delta$-set, and we are done.

      \section{Appendix: Application   to the PWB-method}
      In this section we shall recall the solution to the generalized Dirichlet problem for balayage spaces
      by the Perron-Wiener-Brelot method and how (a~generalization of) Evans theorem enables us to use only
      \emph{harmonic} upper and lower functions.
      
      So let $(X,\W)$ be a balayage space (where the assumption in Section \ref{intro} may be satisfied or not).
      Let $\B(X)$, $\C(X)$ denote the set of all Borel measurable numerical functions on~$X$,
      continuous real functions on $X$, respectively. Let $\Po$ be the set of all \emph{continuous real potentials}
      for $(X,\W)$, that is
\begin{equation*}
  \mathcal P:=\{p\in \W\cap \C(X)\colon \exists\ q\in \W\cap \C(X),\, q>0, \,p/q\to 0 \mbox{ at infinity}\}, 
\end{equation*}
see \cite{BH, H-course} for a thorough treatment and, for example, \cite{ H-compact, H-prag-87}
 for an introduction to balayage spaces.

We recall that, for all  open sets $V$ in $X$ and $x\in X$,  we have  positive Radon
measures $\vx^\vc$  on $X$, supported by $\vc$ and  characterized by
\begin{equation*}
  \int p\,d \vx^\vc=R_p^\vc(x):=\inf\{w(x)\colon w\in \W,\,w\ge p\mbox{ on }V^c\} , \qquad p \in \Po,
  \end{equation*} 
so that, obviously, $\vx^\vc=\delta_x$ if $x\in \vc$. They lead to  \emph{harmonic kernels} $H_V$ on $X$:
  \begin{equation*}
    H_Vf(x):=\int f\,d\vx^\vc, \qquad f\in\B^+(X),\, x\in X.
    \end{equation*} 

    Let us now fix  an open set $U$  in $X$ for which we shall consider the generalized Dirichlet problem
    (see \cite[Chapter VII]{BH}). Let $\V(U)$ denote the set of all open sets~$V$ such that $\ov V$ is compact
    in $U$, and let  $\hypo$ be the set of all   functions $u\in\B(X)$  which are \emph{hyperharmonic on~$U$}, that is,
    are lower semicontinuous on $U$ and satisfy 
    \begin{equation*}
        -\infty<    H_Vu(x)\le u(x)\mbox{  for all  }x\in V\in \V(U).
      \end{equation*}
      Then $\H(U):= \hypo\cap (-\hypo)$
      is the set of    functions which are \emph{harmonic on~$U$},  
    \begin{equation*} 
     \H(U) =\{h\in \B(X)\colon h|_U\in \C(U),\ H_Vh(x)=h(x) \mbox{ for all }x\in V\in\V(U)\}.
      \end{equation*} 

      A function $f\colon X\to \ov \real$ is called \emph{lower $\Po$-bounded}, \emph{$\Po$-bounded}
      if there is some $p\in\Po$ such that
      $f\ge -p$, $|f|\le p$, respectively. For every numerical function~$f$ on~$X$, we have the set of all \emph{upper functions}
    \begin{equation*} 
      \U_f^U:=\{u\in \hypo\colon u\ge f\mbox{ on }\uc,\ u\mbox{  lower $\Po$-bounded and  l.s.c.\ on $X$}\}, 
    \end{equation*}
    the set $  \L_f^U:=-\U_{-f}^U$ of all \emph{lower functions for $f$ with respect to $U$}, and the definitions
       \begin{equation*}\label{bar}
  {\ov H}_f^U :=\inf \,\U_f, \qquad {\underbar H}_f^U:=\sup \,\L_f^U. 
\end{equation*} 

For every $p\in \Po$, there exists $q\in \Po$, $q>0$, such that $p/q\to 0$ at infinity. Hence we may replace $\U_f^U$
by the smaller set of upper functions, which are positive  outside a~compact in $X$,  without changing the infimum
(if $f\ge -p$ consider $f+\ve q$, $\ve>0$). 

To avoid technicalities we state the \emph{resolutivity result}  (see \cite[VIII.2.12]{BH}) only for $\Po$-bounded functions:

     \begin{theorem}\label{pwb}
       For every $\Po$-bounded $f\in\B(X)$,
       \begin{equation*}
         H_Uf={\ov H}_f^U={\underbar H}_ f^U \in \H(U).
         \end{equation*} 
     \end{theorem}

     \begin{remark}\label{harmonic-case}
       {\rm
         Let us indicate how the general approach above yields the solution to the generalized Dirichlet problem for harmonic
         spaces in the way the reader may be more familiar with.

         So let us assume for a moment that the harmonic  measures $\vx^\vc$,
         $x\in V$, for our balayage space  are supported by~$\partial V$ so that (hyper)harmonicity on $U$ does not depend on values on $\uc$,
         and let us identify functions on $U$ with functions on $X$ vanishing outside $U$.

         Let $f$ be a Borel measurable function  on $\partial U$ which is $\Po$-bounded (amounting to boundedness
         if $U$ is relatively compact)
      and let    $\tilde \U_f^U$  be the set of all functions~$u$  on~$U$ which are hyperharmonic on $U$ and  satisfy
  \begin{equation}\label{liminf}
    \liminf\nolimits_{x\in U,\,  x\to z}u(x)\ge f(z)\quad\mbox{  for every }z\in \partial U.
  \end{equation}
  If $u\in \U_f^U$, then $\tilde u:=1_Uu$ is hyperharmonic on $U$ and
  $\liminf\nolimits_{x\to z}\tilde u(x)\ge u(z)\ge f(z)$ for every $z\in \partial U$, hence $\tilde u\in \tilde \U_f^U$.
  If, conversely, $\tilde u$ is a function in $\tilde \U_f^U$ then,   extending it to $X$ by $\liminf_{x\to z} u(z)$  for $z\in\partial U$
  and $\infty$ on $X\setminus \ov U$,  we get a~function $u\in \U_f^U$.
   Therefore Theorem \ref{pwb} yields  that  $h\colon x\mapsto \vx^\uc(f)$, $x\in U$, is harmonic on $U$ and
   \begin{equation*}
     h(x)=\inf \tilde \U_f^U(x)=\sup \tilde \L_f^U(x)  \quad\mbox{  for every } x\in U.
   \end{equation*}
    }
\end{remark}

\medskip

Let {$\ureg$} denote the set of   \emph{regular} boundary points $z$  of $U$, that is, $z\in \partial U$ such that
$\lim_{x\to z} H_Uf(x)=f(z)$ for all $\Po$-bounded   $f\in \C(X)$,   and let $\uirr$ be the set of  \emph{irregular} boundary
points of $U$,  $\uirr:=\partial U\setminus \ureg$.

  \begin{corollary}\label{evans-corollary}
    Suppose that there is   a~lower semicontinuous function~$h_0\ge 0$ on~$X$
      which is harmonic on $U$  and satisfies  $h=\infty$ on~$\uirr$.  
 Then
    \begin{equation*}
       H_Uf=\inf  \,  \U_f^U\cap \H(U) =\sup \,\L_f^U\cap \H(U) \quad\mbox{ for every $\Po$-bounded }  f\in\B(X).
      \end{equation*} 
     \end{corollary}

     \begin{proof}  
       a)  Let $g$ be $\Po$-bounded and lower semicontinuous on $X$. Then there exist $\Po$-bounded  $\vp_n$ in~$\C(X)$, $n\in\nat$,
       such that $\vp_n\uparrow g$. For all $z\in \ureg$ and $n\in\nat$,
       \begin{equation*}
         \liminf\nolimits_{x\to z} H_Ug(x)\ge \liminf\nolimits_{x\to z} H_U\vp_n(x)=\vp_n(z), 
       \end{equation*}
      and hence $  \liminf_{x\to z} H_Ug(x)\ge g(z)$. 
       Clearly,
       \begin{equation*}
         h_n:=H_Ug+(1/n) h_0\in \H(U)
       \end{equation*}
       satisfies $\lim_{x\to z} h_n(x)=\infty$ for all   $z\in\uirr$, and  $h_n$ is lower semicontinuous on~$X$.
        Thus         $           h_n\in  \U_f^U\cap \H(U)$.

        b) Let $f\in \B(X) $ be $\Po$-bounded, $x\in X$.
        There exists a decreasing sequence~$(g_n)$ of  $\Po$-bounded  lower semicontinuous functions
         on $X$  such that $g_n\ge f$ for every  $n\in\nat$ and
         \begin{equation*}
           \int f\, d\vx^\uc=\inf\nolimits_{n\in\nat} \int g_n\,d\vx^\uc,
         \end{equation*}
         that is, $H_Uf(x)=\inf_{n\in\nat}  H_Ug_n(x)$. Hence 
         \begin{equation*}
           H_Uf(x)=\inf\nolimits_{n\in\nat}  (H_U g_n +(1/n)h_0) (x),
         \end{equation*}
         where $H_Ug_n+(1/n)h_0\in  \U_f^U\cap \H(U)$, by (a). Thus $H_Uf=\inf  \U_f^U\cap \H(U)$.

         c) Further,  $H_Uf=-H_U(-f)=- \inf   \U_{-f}^U\cap \H(U)= \sup  \L_f^U\cap \H(U)$.
         \end{proof}

         \begin{remarks}
{\rm
  1. For harmonic spaces, the result in Corollary \ref{evans-corollary} has been proven in \cite{cornea}, where the solution
  to the generalized Dirichlet problem is obtained using \emph{controlled convergence}.

  2. In general, the set $\uirr$ is a semipolar $F_\sigma$-set. Of course, if $(X,\W)$ satisfies Hunt's hypothesis (H), that is,
  if every semipolar set is polar, then $\uirr$ is polar for every $U$.
   Let us note that (H) holds if $X$ is an abelian group such that $\W$ is invariant under translations
  and $(X,\W)$ admits a Green function having the local triangle property (see \cite{HN-(H)}).

  By Theorem \ref{main-first},
  we obtain that in this situation (which covers the classical case, many translation-invariant second order  PDO's
  as well as Riesz potentials, that is, $\a$-stable processes, and many more general L\'evy processes)
  the assumption of Corollary~\ref{evans-corollary} holds. 
  }
\end{remarks}

\bibliographystyle{plain}

{\small \noindent 
Wolfhard Hansen,
Fakult\"at f\"ur Mathematik,
Universit\"at Bielefeld,
33501 Bielefeld, Germany, e-mail:
 hansen$@$math.uni-bielefeld.de}\\
{\small \noindent Ivan Netuka,
Charles University,
Faculty of Mathematics and Physics,
Mathematical Institute,
 Sokolovsk\'a 83,
 186 75 Praha 8, Czech Republic, email:
netuka@karlin.mff.cuni.cz}

\end{document}